\documentclass[a4paper]{amsart}
\usepackage{hyperref}
\usepackage{float,mathrsfs,amsmath,amssymb,amsthm,multirow,tabls,
	enumerate,braket,amscd}
\usepackage{graphicx}
\usepackage[all]{xy}

\usepackage{color}

\numberwithin{equation}{section}

\theoremstyle{plain}
\newtheorem{theorem}{Theorem}[section]

\newtheorem{proposition}[theorem]{Proposition}
\newtheorem{lemma}[theorem]{Lemma}

\theoremstyle{definition}
\newtheorem{definition}[theorem]{Definition}
\newtheorem*{acknowledgements}{Acknowledgements}

\theoremstyle{remark}
\newtheorem{remark}[theorem]{Remark}

\DeclareMathOperator{\Km}{\operatorname{Km}}
\DeclareMathOperator{\Ino}{\operatorname{Ino}}
\DeclareMathOperator{\Hom}{Hom}
\DeclareMathOperator{\rank}{rank}
\DeclareMathOperator{\Gal}{Gal}
\newcommand{\Pic}{\operatorname{Pic}}
\renewcommand{\div}{\operatorname{div}}
\newcommand{\I}{\mathrm{I}}
\newcommand{\II}{\mathrm{II}}

\newcommand{\IV}{\mathrm{IV}}
\newcommand{\Qbar}{{%
\kern.1em\overline{\kern-.1em\mathbf{Q}\kern-.1em}\kern.1em}}
\newcommand{\Obar}{{%
\kern.15em\overline{\kern-.15em O\kern-.0em}\kern.0em}}
\renewcommand{\theenumi}{\roman{enumi}}
\renewcommand{\labelenumi}{\hbox to 1.5em {\upshape(\hfil\theenumi\hfil)}}
\newcommand{\<}{\langle} 
\renewcommand{\>}{\rangle}

\title[Mordell-Weil lattice of Inose's elliptic $K3$ surface]
{Mordell-Weil lattice of Inose's elliptic $K3$ surface arising from the
product of $3$-isogenous elliptic curves}
\date{} 
\author{Masato Kuwata \and Kazuki Utsumi}
\address{Faculty of Economics,
  Chuo University \endgraf 742-1 Hachioji-shi, Tokyo 192-0393 Japan}
\email{kuwata@tamacc.chuo-u.ac.jp}
\address{College of Science and Engineering, Ritsumeikan
  University \endgraf 1-1-1 Noji-higashi, Kusatsu Shiga 525-8577 Japan}
\email{kutsumi@fc.ritsumei.ac.jp}


\begin{document}

\begin{abstract} 
From the product of two elliptic curves, Shioda and Inose \cite{Shioda-Inose} constructed an elliptic $K3$ surface having two $\II^*$ fibers.  Its Mordell-Weil lattice structure depends on the morphisms between the two elliptic curves. In this paper, we give a method of writing down generators of the Mordell-Weil lattice of such elliptic surfaces when two elliptic curves are $3$-isogenous. In particular, we obtain a basis of the Mordell-Weil lattice for the singular $K3$ surfaces $X_{[3,3,3]}$, $X_{[3,2,3]}$ and $X_{[3,0,3]}$.
\end{abstract}

\maketitle

\section{Introduction}

In the study of the geometry, arithmetic and moduli of $K3$ surfaces, elliptic $K3$ surfaces with large Picard number play a vital role. In 1977 Shioda and Inose \cite{Shioda-Inose} gave a classification of singular $K3$ surfaces, that is, $K3$ surfaces with maximum Picard number.  For this purpose, they constructed elliptic $K3$ surfaces $\mathcal{E}$ with two singular fibers of type $\II^{*}$ starting from the Kummer surface $\Km(E_{1}\times E_{2})$ with the product of two elliptic curves $E_{1}$ and $E_{2}$.  They constructed $\mathcal{E}$ as a double cover of $\Km(E_{1}\times E_{2})$ with certain properties (now called a Shioda-Inose structure).  Later, Inose \cite{Inose} gave an explicit model of such an elliptic $K3$ surface as a quartic surface in~$\mathbf{P}^{3}$, and remarked that it is the quotient of $\Km(E_{1}\times E_{2})$ by an involution.  We call the Kodaira-N\'eron model of $\mathcal{E}$ the Inose surface associated with $E_{1}$ and $E_{2}$, and denote it by $\Ino(E_{1}, E_{2})$.  We thus have a ``Kummer sandwich'' diagram:
\[
   \Km(E_{1} \times E_{2}) \overset{\pi_2}{\dashrightarrow}
   \Ino(E_{1}, E_{2}) \overset{\pi_1}{\dashrightarrow}
   \Km(E_{1} \times E_{2})
\] (cf. \cite{Shioda:Kummer-sandwich}).
Also, $\mathcal{E}$ as an elliptic surface with two $\II^{*}$ fibers is denoted by $F^{(1)}_{E_{1},E_{2}}$.  This notation reflects that it is a part of the construction of elliptic $K3$ surfaces of high rank by the first named author \cite{Kuwata:MW-rank}, where he constructed $F^{(n)}_{E_{1},E_{2}}$, $n=1,\dots,6$, which has various Mordell-Weil rank up to~$18$.

The structure of the Mordell-Weil lattice of $F^{(1)}_{E_{1},E_{2}}$ is known to be isomorphic to $\Hom(E_{1},E_{2})\<2\>$ if $E_{1}$ and $E_{2}$ are nonisomorphic (see \cite{Shioda:correspondence}).  Here, for a lattice $L$, we denote by $L\<n\>$ the lattice structure on $L$ with the pairing multiplied by~$n$.  However, given an isogeny $\varphi \in \Hom(E_{1},E_{2})$ and the Weierstrass equation of $F^{(1)}_{E_{1},E_{2}}$, it is quite difficult to write down the coordinates of the section corresponding to~$\varphi$, and it has been worked out only in limited cases (cf.~\cite{Shioda:F^(5)}, \cite{Kumar-Kuwata}). Most known examples fall into the case where the degree of isogeny $\varphi$ equals~$2$, in which case the calculations are straight forward. One particular example of the case $\deg\varphi=4$ is dealt in \cite[Example~9.2]{Kumar-Kuwata}. In this paper we consider a family of the pairs of elliptic curves $E_{1}$ and $E_{2}$ with an isogeny $\varphi:E_{1}\to E_{2}$ of degree~$3$ defined over~$k$.  We write down a formula of the section of $F^{(1)}_{E_{1},E_{2}}$ coming from $\varphi$ defined over the base field~$k$.  To do so, we first work with the surface $F^{(6)}_{E_{1},E_{2}}$, which has a simple affine model that can be viewed as a family of cubic curves with a rational point over~$k$.  We modify the method in \cite{Kumar-Kuwata} to find sections of $F^{(1)}_{E_{1},E_{2}}$.  We also give a section of $F^{(2)}_{E_{1},E_{2}}$ coming from the isogeny~$\varphi$, and give a basis defined over the field $k(E_{1}[2],E_{2}[2])$ when $E_{1}$ and $E_{2}$ do not have a complex multiplication.

In \S\ref{singular} we study some examples of singular $K3$ surfaces in detail.  In particular, we determine a basis of the MWL of the Inose surface $F_{E_{1},E_{2}}^{(1)}$ and that of $F_{E_{1},E_{2}}^{(2)}$ for the singular $K3$ surfaces $X_{[3,3,3]}$, $X_{[3,2,3]}$ and $X_{[3,0,3]}$ which correspond to the quadratic forms $3x^{2}+3xy+3y^{2}$, $3x^{2}+2xy+3y^{2}$, and $3x^{2}+3y^{2}$ respectively.

\begin{acknowledgements}
We would like to express sincere gratitude to Professor Ichiro Shimada for useful discussions. We would also like to thank Professor Hisanori Ohashi for his helpful comments. Furthermore, we would like to thank the referee for his/her profitable comments and corrections. The computer algebra system \verb|Maple| was used in the calculations for this paper. Kuwata was partially supported by JSPS KAKENHI Grant Number JP26400023, and by the Chuo University Grant for Special Research. Utsumi was partially supported by the Ritsumeikan University Research Promotion Program for Aquiring Grants in-Aid for Scientific Research.
\end{acknowledgements}

\section{Inose surface}\label{Inose}

Throughout this paper the base field~$k$ of algebraic varieties is assumed to be a number field.

Let $\Km(E_{1} \times E_{2})$ be the Kummer surface associated with the product of elliptic curves $E_{1}$ and $E_{2}$, that is, the minimal resolution of the quotient surface $E_{1} \times E_{2} / \{ \pm 1 \}$. If $E_{1}$ and $E_{2}$ are defined by the equations
\[
	E_{1} : y_1^2 = f_1(x_1), \quad E_{2} : y_2^2 = f_2(x_2),
\]
where $f_1(x_1)$ and $f_2(x_2)$ are cubic polynomials, an affine singular model of $\Km(E_{1} \times E_{2})$ is given as the hypersurface in $\mathbf{A}^{3}$ defined by the equation
\begin{equation}\label{eq:Cubic}
	f_2(x_2)=t^2 f_1(x_1),
\end{equation}
where $t= y_2 / y_1$. Then, the map $\Km(E_{1} \times E_{2}) \to \mathbf{P}^1$ induced by $(x_1, x_2, t) \mapsto t$ is a Jacobian fibration, which is sometimes called Inose's pencil (cf.~\cite{Kuwata-Shioda}).

Take a parameter $u$ such that $t=u^3$, and consider \eqref{eq:Cubic} as a family of cubic curves in $\mathbf{P}^{2}=\{(x_{1}:x_{2}:z)\}$, or a cubic curve over $k(u)$. Then, we see that it has a rational point $(x_{1}:x_{2}:z)=(1:u^2:0)$.  Using this point as the origin, we consider it as an elliptic curve over $k(u)$. In particular, if $f_{1}(x_{1})$ and $f_{2}(x_{2})$ are given by
\[
f_{1}(x_{1})=x_{1}^3+a_{2}x_{1}^{2}+a_{4}x_{1}+a_{6},\quad
\text{and}\quad
f_{2}(x_{2})=x_{2}^3+a'_{2}x_{2}^{2}+a'_{4}x_{2}+a'_{6},
\]
then we can convert \eqref{eq:Cubic} to the Weierstrass form 
\[
Y^2 = X^3 - \frac{1}{3} A\,X
	+ \frac{1}{64}\Bigl(\Delta_{E_{1}}u^{6}
		+B+\frac{\Delta_{E_{2}}}{u^{6}}\Bigr),
\]
where
\[
\left\{
\begin{aligned}
&A = (a_{2}^2-3a_{4})({a'_{2}}^2-3a'_{4}), \\
&B = \frac{32}{27}(2a_{2}^3-9a_{2}a_{4}+27a_{6})(2{a'_{2}}^3-9a'_{2}a'_{4}+27a'_{6}), \\
&\Delta_{E_{1}} = 16
(a_{2}^2a_{4}^2-4a_{2}^3a_{6}+18a_{2}a_{4}a_{6}-4a_{4}^3-27a_{6}^2), \\
&\Delta_{E_{2}} =16({a'_{2}}^2{a'_{4}}^2-4{a'_{2}}^3a'_{6}+18a'_{2}a'_{4}a'_{6}
-4{a'_{4}}^3-27{a'_{6}}^2).
\end{aligned}
\right.
\]

Let $s=t^{2}=u^{6}$.  Define $F^{(1)}_{E_{1} E_{2}}$ to be
\begin{equation}\label{eq:F1}
	F^{(1)}_{E_{1} E_{2}} :
	Y^2 = X^3 - \frac{1}{3} A\,X 
	+ \frac{1}{64}\Bigl(\Delta_{E_{1}}s
		+B+\frac{\Delta_{E_{2}}}{s}\Bigr).
\end{equation}
This Jacobian fibration has two reducible fibers of type $\II^*$ at $s=0$ and $s=\infty$. \begin{definition}[cf.~\cite{Kumar-Kuwata}] The Kodaira-N\'{e}ron model of the elliptic surface $F^{(1)}_{E_{1},E_{2}}$ over $k$ defined by (\ref{eq:F1}) is called the Inose surface associated with $E_{1}$ and $E_{2}$, and it is denoted by $\Ino(E_{1}, E_{2})$.
\end{definition}

\begin{remark}
Over a certain extension of~$k$, the equation of $F^{(1)}_{E_{1}, E_{2}}$ may be given by
\[
F^{(1)}_{E_{1} E_{2}} : 
Y^2 = X^3  -  3 \sqrt[3]{J_1 J_2} \ X 
	+ s+\frac{1}{s} - 2\sqrt{(1-J_1)(1-J_2)},
\]
where $J_i = j(E_{i})/1728$ ($i=1,2$) with $j(E_{i})$ the $j$-invariant of $E_{i}$ (cf. \cite{Inose}, \cite{Shioda:correspondence}).
\end{remark}

\begin{definition}
For $n \geq 1$, the elliptic surface $F^{(n)}_{E_{1}, E_{2}}$ over $k$ is defined by
\begin{equation}\label{eq:Fn}
F^{(n)}_{E_{1} E_{2}} :
Y^2 = X^3 - \frac{1}{3} A\,X 
	+ \frac{1}{64}\Bigl(\Delta_{E_{1}}s^{n}
		+B+\frac{\Delta_{E_{2}}}{s^{n}}\Bigr).
\end{equation}
\end{definition}

\begin{remark}
(1) 
The Kodaira-N\'{e}ron model of $F^{(n)}_{E_{1}, E_{2}}$ is a $K3$ surface for $n=1, \ldots, 6$, but not for $n \geq 7$ (\cite{Kuwata:MW-rank}).

(2) 
Since $u^{6}=t^{2}$, Inose's pencil on $\Km(E_{1}\times E_{2})$ is isomorphic to $F^{(2)}_{E_{1},E_{2}}$.  However, the isomorphism between \eqref{eq:Cubic} and \eqref{eq:Fn} for $n=2$ may not be defined over $k$ itself.
\end{remark}

\section{Mordell-Weil lattice of the Inose surface}\label{MWL}

In this section we give a summary of known facts on the Mordell-Weil lattice of $F^{(1)}_{E_{1},E_{2}}$ over~$\bar k(s)$.

\begin{theorem}\cite[Theorem\ 6.3]{Shioda:correspondence} \label{thm:MWrank}
The Jacobian fibration $F^{(1)}_{E_{1}, E_{2}}$ on $\Ino(E_{1},E_{2})$ has two singular fibers of type $\II^*$ at $s=0$ and $\infty$, and the other singular fibers and its Mordell-Weil rank are given in the table below. Here, $j_{i}=j(E_{i})$, $i=1,2$, are the $j$-invariants, and $h$ is the rank of $\Hom_{\bar k}(E_{1}, E_{2})$.
\[
\begin{tabular}{l|c|c}
\hline
\hfil $j$-invariants & singular fibers & Mordell-Weil rank\\
\hline
$j_1 \neq j_2, j_1 j_2 \neq 0$ & $4\,\I_1$ & $h$\\
$j_1 \neq j_2, j_1j_2=0$ & $2\,\II$  & $h$ \\
$j_1=j_2 \neq 0,1728$ & $\I_2, 2\,\I_1$  & $h-1$ \\
$j_1=j_2=1728$ & $2\,\I_2$  & $0$\\
$j_1=j_2=0$ & $\IV$ & $0$ \\
\hline
\end{tabular}
\]
Assume $j_1 \neq j_2$. Then, the Mordell-Weil group $F^{(1)}_{E_{1}, E_{2}}(\Qbar(s))$ is torsion-free, and isomorphic to the lattice $\Hom_{\bar k}(E_{1},E_{2})\<2\>$, where the pairing of $\Hom_{\bar k}(E_{1}, E_{2})$ is given by
\[
(\varphi, \psi) = \frac{1}{2}\left( \deg(\varphi+\psi) -\deg
\varphi -\deg \psi \right) \quad \varphi, \psi \in \Hom_{\bar k}(E_{1},
E_{2}). 
\] 
\end{theorem}

\begin{remark}
The notation $\<n\>$ means that the paring of the lattice is multiplied by~$n$.
\end{remark}

\section{3-isogenies of elliptic curves}\label{3iso}

We recall some general facts on $3$-isogenies between elliptic curves following J.~Top \cite[\S3]{Top}.  

\begin{lemma}\label{lem:3-isog}
Let $E$ be an elliptic curve over a number field $k$, and $G\subset E(\bar k)$ a subgroup of order three that is stable under the action of $\Gal(\bar k/k)$. 
Then, the pair of $E$ and $G$ is one of the following:
\begin{itemize}
\item[(i)]
$E$ is given by $y^{2}=x^{3}+d$ and $G$ is generated by $P=(0,\sqrt{d})$.
\item[(ii)]
$E$ is given by $y^{2}=x^{3}+a(x-b)^{2}$ and $G$ is generate by $P=(0,b\sqrt{a})$.
\end{itemize}
\end{lemma}

\begin{proof}
$E$ can be given by an equation of the form $y^{2}=f(x)$ with $\deg f=3$.  Then, $G$ consists of three points $O=(0:1:0)$, $P=(\alpha,\beta)$, and $2P=-P=(\alpha,-\beta)$ with $\beta\neq0$.  Since $G$ is Galois invariant, we have $\alpha\in k$ and $\beta^{2}\in k$.  Replacing $x$ by $x+\alpha$ if necessary, we may assume $\alpha=0$.  The curve is now given by an equation $y^{2} = x^{3} + ax^{2} + cx + d$, and $P=(0,\sqrt{d})$.  The tangent line at $P$ is given by $y=cx/(2\sqrt{d})+\sqrt{d}$, and this tangent line intersects with $E$ at $P$ with multiplicity~$3$ if and only if $c^{2} = 4ad$.
If $c = 0$, our equation is $y^{2} = x^{3} + d$. If $c\neq 0$, the equation can be written as $y^{2} = x^{3} + a(x - b)^{2}$, and $P=(0,b\sqrt{a})$.
\end{proof}

If $E$ is given by $y_{1}^{2}=x_{1}^{3}+d$ and $P=(0,\sqrt{d})$, then the quotient $E/G$ is given by the equation 
\[
E/G: y_{2}^{2}=x_{2}^{3}-27d,
\]
and the quotient map $\varphi:E\to E/G$ is given by
\[
\varphi:(x_{1},y_{1})\mapsto (x_{2},y_{2})
=\left(\frac{x_{1}^{3}+4d}{x_{1}^{2}},
\frac{(x_{1}^{3}-8d)y_{1}}{x_{1}^{3}}\right).
\]
In this case $E$ and $E/G$ are isomorphic over $\bar k$, and the $j$-invariants are equal to~$0$.

If $E$ is given by
\[
  E:  y_1^2=x_1^3+a(x_1-b)^2,
\]
and $P=(0, b\sqrt{a})$, then the quotient $E/G$ is given by
\[
  E/G : y_2^2=x_2^3-3a\bigl(x_2-(4a+27b)/9\bigr)^2.
\]
The isogeny $\varphi : E \to E/G$ is given by $\varphi(x_1, y_1) = (\varphi_x(x_1), \varphi_y(x_1) y_1)$, where
\begin{equation}
\label{eq:phi}
\varphi_x(x_1)=\frac{3x_1^3+4ax_1^2-12abx_1+12ab^2}{3x_1^2}, \quad
\varphi_y(x_1)=-\frac{x_1^3+4abx_1-8ab^2}{x_1^3}.
\end{equation}

\section{The rational section of $F_{E_{1},E_{2}}^{(1)}$ arising from a 3-isogeny}\label{sec:F1}

In this section, we assume $E_{1}$ and $E_{2}$ are $3$-isogenous over~$k$, and we find an explicit section of $F_{E_{1},E_{2}}^{(1)}$.

In the case where the $j$-invariants of both $E_{1}$ and $E_{2}$ are equal to~$0$, the Mordell-Weil group $F_{E_{1},E_{2}}^{(1)}(\bar k(s))$ is trivial by Theorem~\ref{thm:MWrank}, and we have nothing to do.

As for the second case in \S4, suppose that $E_{1}$ and $E_{2}$ are given by
\begin{equation}\label{eq:E1andE2}
\begin{aligned}
&E_{1}:y_{1}^{2}=x_{1}^{3}+a(x_{1}-b)^2, \\
&E_{2}:y_{2}^{2}=x_{2}^{3}+a'\bigl(x_{2}-b'\bigr)^2,
\end{aligned}
\end{equation}
where $a,b\in k$, $a'=-3a$, and $b'=(4a+27b)/9$. We work with the cubic curve over $k(u)$ in $\mathbf{P}^{2}=\{(x_{1}:x_{2}:z)\}$ given by
\begin{equation}\label{eq:C_u}
C_{u}:x_{2}^{3}+a'\bigl(x_{2}-b'z\bigr)^2z
	=u^{6}\bigl(x_{1}^{3}+a(x_{1}-bz)^2z\bigr),
\end{equation}
which is isomorphic over $k(u)$ to $F_{E_{1},E_{2}}^{(6)}$ with the choice of origin $O=(1:u^{2}:0)$. Its Weierstrass equation is given by
\begin{multline}\label{eq:WeierF6}
F_{E_{1},E_{2}}^{(6)}:Y^{2} = X^3
-\frac{1}{3}aa'(a+6b)(a'+6b')X
\\
-\frac{1}{4}\left(a^{2}b^{3}(4a+27b)u^{6} 
+\frac{a'^{2}b'^{3}(4a'+27b')}{u^{6}}\right)
\\
+\frac{1}{54}aa'
\bigl(3(a+3b)^2-a^{2}\bigr)\bigl(3(a'+3b')^2-a'^{2}\bigr).
\end{multline}
The change of coordinates are given by
\begin{equation}\label{eq:change-var}
\left\{\begin{aligned}
X&=\frac{c_{6}u^{6}+c_{4}u^{4}+c_{2}u^{2}+c_{0}}%
{3u^2((3x_{1}+az)u^2-(3x_{2}+a'z))},
\\[\medskipamount]
Y&=\frac{d_{10}u^{10}+d_{6}u^{6}+d_{4}u^{4}+d_{0}}%
{6u^3\bigl((3x_{1}+az)u^2-(3x_{2}+a'z)\bigr)^2},
\end{aligned}\right.
\end{equation}
where
\begin{gather*}
\hbox to 0.87\textwidth{
$c_{6}=\ 2a(a+6b)(3x_{1}+az)-a\bigl(3(a+3b)^2-a^{2}\bigr)z,$\hfil}
\\
\hbox to 0.87\textwidth{
$c_{4}=\ \ a(a+6b)(3x_{2}+a'z),$\hfil}
\\
\hbox to 0.87\textwidth{
$c_{2}=-a'(a'+6b')(3x_{1}+az),$\hfil}
\\
\hbox to 0.87\textwidth{
$c_{0}=-2a'(a'+6b')(3x_{2}+a'z)+a'\bigl(3(a'+3b')^2-a'^{2}\bigr)z,$\hfil}
\\[\medskipamount]
\hbox to 0.87\textwidth{
$d_{10}=
-a\bigl(3(a+3b)^2-a^{2}\bigr)\bigl((3x_{1}+az)^2+2a(a+6b)z^2\bigr)$\hfil}
\\
\hbox to 0.87\textwidth{\hfil 
$+6a^2(a+6b)^2(3x_{1}+az)z,$}
\\
\hbox to 0.87\textwidth{
$d_{6}\,=\,a\bigl(3(a+3b)^2-a^{2}\bigr)\bigl((3x_{2}+a'z)^2+2a'(a'+6b')z^2\bigr)$\hfil}
\\
\hbox to 0.87\textwidth{\hfil
$-\,6aa'(a+6b)(a'+6b')(3x_{1}+az)z,$}
\\
\hbox to 0.87\textwidth{
$d_{4}\,=\,a'\bigl(3(a'+3b')^2-a'^{2}\bigr)\bigl((3x_{1}+az)^2+2a(a+6b)z^2\bigr)$\hfil}
\\
\hbox to 0.87\textwidth{\hfil
$-6aa'(a+6b)(a'+6b')(3x_{2}+a'z)z,$}
\\
\hbox to 0.87\textwidth{
$d_{0}\,=-a'\bigl(3(a'+3b')^2-a'^{2}\bigr)\bigl((3x_{2}+a'z)^2+2a'(a'+6b')z^2\bigr)$\hfil}
\\
\hbox to 0.87\textwidth{\hfil
$+6a'^2(a'+6b')^2(3x_{2}+a'z)z.$}
\end{gather*}

\begin{remark}
The origin $O$ in our case is not an inflection point of the cubic curve $C_{u}$.  Thus, three collinear points $P$, $Q$, $R\in C_{u}$ do not satisfy the equation $P+Q+R=O$ under the group law. Instead, we have  $P+Q+R=\Obar$, where $\Obar$ is the third point of intersection between $C_{u}$ and the tangent line at~$O$.
\end{remark}
 
\begin{remark}\label{rmk:involution}
By definition, the surface $F_{E_{1},E_{2}}^{(1)}$ is obtained as the quotient of \eqref{eq:WeierF6} by the automorphism $(X,Y,u)\mapsto (X,Y,-\omega u)$, where $\omega$ is a third root of unity. It should be noted that the automorphism $\bigl((x_{1}:x_{2}:z),u\bigr)\mapsto \bigl((x_{1}:x_{2}:z),-\omega u\bigr)$ on $C_{u}$ does not correspond to this automorphism since the quotient by the latter gives a rational surface. 
\end{remark}

Replacing $x_{2}$ in \eqref{eq:C_u} by $\varphi_{x}(x_{1})$ given by \eqref{eq:phi} , we obtain
\[
\frac{1}{x_{1}^{6}}\bigl(x_1^3+a(x_1z-bz)^2z\bigr)
\left((x_1^3+4abx_1z^{2}-8ab^2z^{3})^{2}-(u^{3}x_{1}^{3})^{2}\right)
=0.
\]
Define two homogeneous polynomials in $x_{1}, z$ by
\begin{align*}
&p_{\varphi}^{+}(x_{1},z)=x_1^3+4abx_{1}z^{2}-8ab^2z^{3}-u^{3}x_{1}^{3}, 
\\
&p_{\varphi}^{-}(x_{1},z)=x_1^3+4abx_{1}z^{2}-8ab^2z^{3}+u^{3}x_{1}^{3}.
\end{align*}
Let $Q_{1}^{+},Q_{2}^{+}$, and $Q_{3}^{+}$ be points on $C_{u}$ of the form $(x_{1},\varphi_{x}(x_{1}))$ whose $x_{1}$-coordinates  are the three roots of $p_{\varphi}^{+}(x_{1},z)=0$.  Define $D_{\varphi}^{+}=Q_{1}^{+}+Q_{2}^{+}+Q_{3}^{+}$.  We look for a conic $\mathcal{Q}^{+}$ given by
\begin{equation}\label{eq:quadratic}
\mathcal{Q}^{+}:
q^{+}(x_{1},x_{2},z)
=c_{1}x_{1}^{2}+c_{2}x_{1}x_{2}+c_{3}x_{2}^{2}
+c_{4}x_{1}z+c_{5}x_{2}z+c_{6}z^{2}=0
\end{equation}
which passes through $Q_{1}^{+},Q_{2}^{+},Q_{3}^{+}$, and is tangent to the line $\mathcal{T}$ of $C_{u}$ at the origin~$O$ (see Figure~\ref{fig:1}).

\begin{figure}[t]
  \centering
  \includegraphics[width=\textwidth]{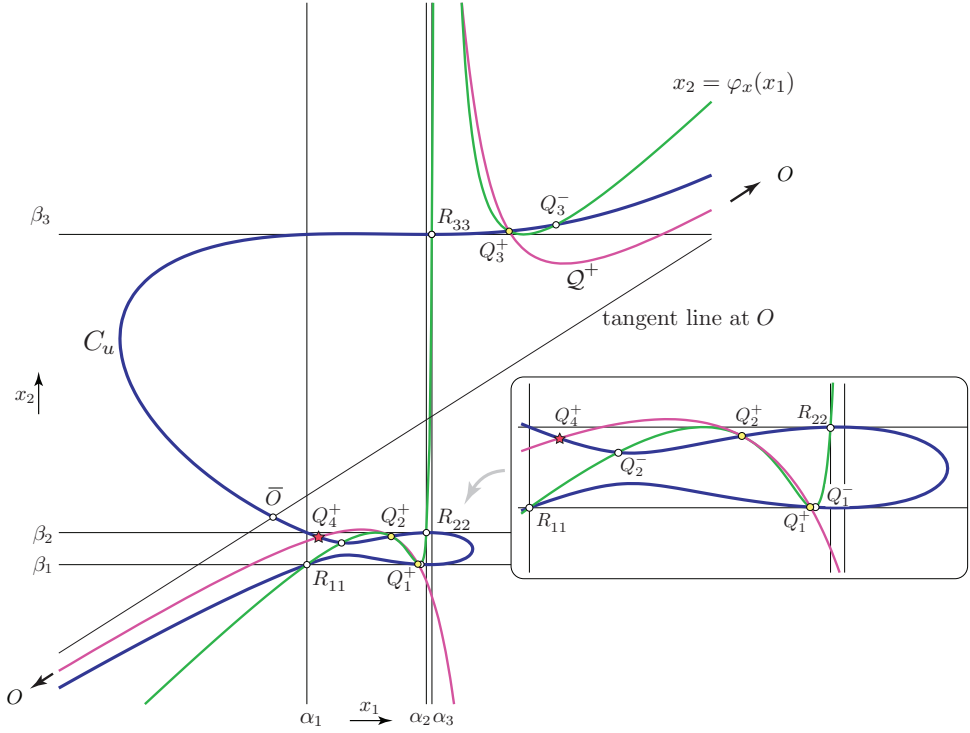}
  \caption{The $x_1$-coordinate of $Q_{i}^{\pm}$ ($i=1,2,3$) is a solution of the cubic equation $\varphi_y(x_1) = \pm u^{3}$ in $\overline{k(u)}$. A conic $\mathcal{Q}^{+}$, which passes through $Q_1^{+}$, $Q_2^{+}$, $Q_3^{+}$, and tangent to the cubic $C_{u}$ \eqref{eq:C_u} at~$O$, intersects with $C_{u}$ at the sixth point $Q_4^{+}$.}
  \label{fig:1}
\end{figure}

The tangent line $\mathcal{T}$ is given by
\begin{equation}\label{eq:tangent}
\mathcal{T} : l(x_{1},x_{2},z)=3u^2x_{1}-x_{2}+a(u^2+3)z=0.
\end{equation}
The condition that  $\mathcal{T}$ is tangent to $\mathcal{Q}^{+}$ is obtained as follows.  Regard $q^{+}(x_{1},x_{2},z)$ and $l(x_{1},x_{2},z)$ as polynomials in $x_{2}$, and compute the remainder of $q^{+}(x_{1},x_{2},z)$ divided by $l(x_{1},x_{2},z)$. Then, $\mathcal{T}$ is tangent to $\mathcal{Q}^{+}$ at $O$ if and only if the coefficients of $x_{1}z$ and $x_{1}^{2}$ vanish. Similarly, $\mathcal{Q}^{+}$ passes through $Q_{1}^{+},Q_{2}^{+},Q_{3}^{+}$ if and only if $q^{+}(x_{1},\varphi_{x}(x_{1}),z)$ is divisible by $p_{\varphi}^{+}(x_{1},z)$.  We then obtain a system of five homogeneous linear equations in $c_{1},\dots,c_{6}$.  It is easy to solve this system to find a conic $\mathcal{Q}^{+}$. By B\'ezout's theorem, the intersection $\mathcal{Q}^{+}$ and $C_{u}$ consists of six points counting multiplicity.  It is not difficult to find the sixth point $Q_{4}^{+}$, but the result is rather complicated to show here.

Similarly, we obtain $Q_{4}^{-}$ starting from the divisor $D_{\varphi}^{-}=Q_{1}^{-}+Q_{2}^{-}+Q_{3}^{-}$ obtained from $p_{\varphi}^{-}(x_{1},z)$.

\begin{proposition}
Let $D_{\varphi}^{+}$
\textup{(}resp.~$D_{\varphi}^{-}$\textup{)} be the divisor on the cubic curve $C_{u}$ defined by the equation $p_{\varphi}^{+}(x_{1},z)=0$ \textup{(}resp.~$p_{\varphi}^{-}(x_{1},z)=0$\textup{)}.
\begin{enumerate}[{\rm (i)}]
\item The divisor $D_{\varphi}^{+}$ \textup{(}resp.~$D_{\varphi}^{-}$\textup{)} determines a $k$-rational point $P_{\varphi}^{+}$ \textup{(}resp.~$P_{\varphi}^{-}$\textup{)} in
$F^{(6)}_{E_{1}, E_{2}}(k(u))$.
\item $P_{\varphi}^{+} - P_{\varphi}^{-}$ is in the image of $F^{(1)}_{E_{1}, E_{2}}(k(s)) \to F^{(6)}_{E_{1}, E_{2}}(k(u))$ induced by $s \mapsto u^6$. The height of its pre-image in $F^{(1)}_{E_{1},E_{2}}(k(s))$ is $6$.
\end{enumerate}
\end{proposition}

\begin{proof}
Since $D_{\varphi}^{+}$ and $D_{\varphi}^{-}$ are both defined over~$k(u)$, $Q_{4}^{+}$ and $Q_{4}^{-}$ are $k(u)$-rational points on $C_{u}$. 
Let $\Psi_{u}:C_{u}\to F^{(6)}_{E_{1}, E_{2}}$ be the isomorphism over $k(u)$ defined by the formula \eqref{eq:change-var}, and $P_{\varphi}^{+}$ (resp.~$P_{\varphi}^{-}$) be the point in $F^{(6)}_{E_{1}, E_{2}}(k(u))$ given by $\Psi_{u}(Q_{4}^{+})$ (resp.~$\Psi_{u}(Q_{4}^{-}$)). 
Let $\sigma$ be the automorphism of ${\bar k}(u)$ defined by $u\mapsto -\omega u$.  It induces an automorphism of $C_{u}(\bar k(u))$ and that of $F_{E_{1},E_{2}}^{(6)}(\bar k(u))$.  We show that the $k(u)$-rational point $P_{\varphi}^{+}-P_{\varphi}^{-}\in F^{(6)}_{E_{1}, E_{2}}(k(u))$ is invariant under~$\sigma$.  This proves that $P_{\varphi}^{+}-P_{\varphi}^{-}$ belongs to the image of $F^{(1)}_{E_{1}, E_{2}}(k(s))$ under the map $s\mapsto u^{6}$.

First, consider the automorphism $\sigma^{3}:u\mapsto -u$.  The explicit conversion formula \eqref{eq:change-var} shows that $\sigma^{3}(\Psi_{u}(Q))=\Psi_{-u}(\sigma^{3}(Q))=-\Psi_{u}(\sigma^{3}(Q))$ for $Q\in C_{u}(k(u))$.  Since $\sigma^{3}$ exchanges $Q_{4}^{+}$ and $Q_{4}^{-}$ by definition, we have $\sigma^{3}(P_{\varphi}^{+})=-P_{\varphi}^{-}$, and $\sigma^{3}(P_{\varphi}^{-})=-P_{\varphi}^{+}$.  This implies that $\sigma^{3}$ leaves $P_{\varphi}^{+}-P_{\varphi}^{-}$ invariant.

Next, consider the automorphism $\sigma^{4}:u \mapsto \omega u$.  Since $\sigma^{4}$ leaves $p_{\varphi}^{+}(x_{1},z)$ (resp.~$p_{\varphi}^{-}(x_{1},z)$) invariant, $D_{\varphi}^{+}$ (resp.~$D_{\varphi}^{-}$) is invariant under~$\sigma^{4}$.  By construction, the divisors $D_{\varphi}^{+}+Q_{4}^{+}$ and $D_{\varphi}^{-}+Q_{4}^{-}$ are linearly equivalent by the function $q^{+}(x_{1},x_{2},z)/q^{-}(x_{1},x_{2},z)$.  This implies that the divisor class $[Q_{4}^{+}-Q_{4}^{-}]=[D_{\varphi}^{-}-D_{\varphi}^{+}]\in \Pic_{\bar k}^{0}(C_{u})$ is invariant under~$\sigma^{4}$.  
The Mordell-Weil group $F^{(6)}_{E_{1}, E_{2}}(\bar k(u))$ can be identified with $\Pic_{\bar k}^{0}(F^{(6)}_{E_{1}, E_{2}})$ by the map $P\mapsto [P-O]$, where $[\quad]$ stands for the divisor class.  Since $\Psi_{u}$ induces an isomorphism of groups $\Pic_{\bar k}^{0}(C_{u})\to\Pic_{\bar k}^{0}(F^{(6)}_{E_{1}, E_{2}})$, we can identify $F^{(6)}_{E_{1}, E_{2}}(\bar k(u))$ with $\Pic_{\bar k}^{0}(C_{u})$ via the map $P\mapsto [\Psi_{u}^{-1}(P)-O]$.  Since $P_{\varphi}^{+}-P_{\varphi}^{-}=[(\Psi_{u}^{-1}(P_{\varphi}^{+})-O)-(\Psi_{u}^{-1}(P_{\varphi}^{-})-O)]=[Q_{4}^{+}-Q_{4}^{-}]$, we see that $P_{\varphi}^{+}-P_{\varphi}^{-}$ is invariant under~$\sigma^{4}$. 

We thus conclude that $P_{\varphi}^{+}-P_{\varphi}^{-}$ is invariant under the automorphism $\sigma=\sigma^{3}\circ\sigma^{4}$.

The calculation of the height is the same as in \cite{Shioda:correspondence}.
\end{proof}

We denote by $P^{(1)}_{\varphi}$ the section $P_{\varphi}^{+}-P_{\varphi}^{-}$ in $F^{(1)}_{E_{1}, E_{2}}(k(s))$. Carrying out calculations according to the above recipe, we obtain the coordinates of $P^{(1)}_{\varphi}$ as follows:
\begin{multline*}
P^{(1)}_{\varphi}=
\left(-\frac{3 S^{3} + c_{2} S^{2} + c_{1} S + c_{0}}
{16aa'\bigl(S-9bb'\bigr)},\right.
\\
\left.-\frac{s\bigl(bs+9b'\bigr)\bigl(9S^{4} + d_{3}S^{3} + d_{2} S^{2} + d_{1} S + d_{0}\bigr)}
{288a^{3}\bigl(bs-9b'\bigr)^{3}}\right),
\end{multline*}
where \allowdisplaybreaks
\begin{align*}
S &=\frac{9}{2}\Bigl(\frac{b^2s}{9}+\frac{9b'^2}{s}\Bigr), \\
c_{2} &= 8aa'+81bb', \quad
c_{1} = \frac{16}{9}a^{2}a'^{2} - 144aa'bb'+ 729 b^2b'^{2}, \\
c_{0} &= bb'(80a^{2}a'^{2} - 1944aa'bb'+ 2187b^2b'^{2}),
\\
d_{3} &= 36\bigl(aa'+9bb'\bigr), \quad
d_{2} = 2(16a^2a'^{2}-162aa'bb'+2187b^2b'^{2}), \\
d_{1} &= -108bb'\bigl(8a^2a'^{2}+135aa'bb'-243b^2b'^{2}\bigr), \\
d_{0} &= -3bb'\bigl(128a^3a'^{3}-6912a^{2}a'^{2}bb'+26244aa'b^2b'^{2}-19683b^3b'^{3}\bigr). 
\end{align*}

\section{Rational sections of $F_{E_{1},E_{2}}^{(2)}$}\label{sec:F2}

Next, we consider the Mordell-Weil lattice $F_{E_{1},E_{2}}^{(2)}$.  We use the same notation as the previous section.

Let $\alpha_{1},\alpha_{2},\alpha_{3}$ be the three roots of $x_{1}^{3}+a(x_{1}-b)^2=0$.  For $i=1,2,3$, define $\beta_{i}=\varphi_{x}(\alpha_{i})$.  Since $\varphi$ is an isogeny of odd degree, $(\beta_{i},0)$ are the $2$-torsion points of $E_{2}$.  Thus, we have $k(E_{1}[2],E_{2}[2])=k(\alpha_{1},\alpha_{2},\alpha_{3})$, which we denote by~$k_{2}$ for short.

Let $R_{ij}$ be the point $(\alpha_{i}:\beta_{j}:1)$ in $C_{u}$.  The quotient of $C_{u}$ by the action $u\mapsto \omega u$ is nothing but the model of Kummer surface \eqref{eq:Cubic}. 

\def\ShiodaTh{\cite[Theorem~1.2]{Shioda:correspondence}}
\begin{theorem}[cf. \ShiodaTh ]
\label{thm:F2-Shioda}
Suppose $j(E_{1})\neq j(E_{2})$, and let $h$ be the rank of $\Hom_{\bar k}(E_{1},E_{2})$.   Then, the Mordell-Weil lattice $F^{(2)}_{E_{1}, E_{2}}(\bar k(t))$ contains a sublattice of index $2^{h}$ naturally isomorphic to 
\[
\Hom_{\bar k}(E_{1},E_{2})\<4\>\oplus A_{2}^{*}\<2\>^{\oplus 2},
\]
where $A_{2}^{*}$ denotes the dual lattice of the root lattice $A_{2}$, and $\<n\>$ denotes the lattice with the pairing multiplied by~$n$. 
In particular, the determinant of the height matrix of $F^{(2)}_{E_{1}, E_{2}}(\bar k(t))$ equals $2^{4}/3^{2}$ times the determinant of $\Hom_{\bar k}(E_{1},E_{2})$.
\end{theorem}

\begin{proof}
Except for the last statement, it is Theorem~1.2 of \cite{Shioda:correspondence}.  The last statement follows from the fact that if $L'\subset L$ is a sublattice of finite index in $L$, then we have $\det L'=\det L \times [L:L']^{2}$. 
\end{proof}

In case $\Hom_{\bar k}(E_{1},E_{2})=0$, it is known (cf. \cite[Prop.~3.3]{Kumar-Kuwata}) that, by taking $R_{11}$ as the origin, $F_{E_{1},E_{2}}^{(2)}(\bar k(t))$ is generated by $R_{22},R_{33},R_{23},R_{32}$, and is isomorphic to the lattice $A_{2}^{*}\<2\>$.

Recall that the origin $O$ in $C_{u}$ is $(x_{1}:x_{2}:z)=(1:u^{2}:0)$.  Let $O_{\omega}$ and $O_{\omega^{2}}$ be the points in $C_{u}$ defined by $\bigl(1:(\omega u)^{2}:0\bigr)$ and $\bigl(1:(\omega^{2} u)^{2}:0\bigr)$ respectively.  Clearly, the action $u\mapsto \omega u$ induces a cyclic permutation $O\mapsto O_{\omega}\mapsto O_{\omega^{2}}$.  Thus, the divisor $O+O_{\omega}+O_{\omega^{2}}$ is invariant under this action. Let $q(x_{1},x_{2},z)$ be the quadratic form in \eqref{eq:quadratic} and $l(x_{1},x_{2},z)$ the linear form  in \eqref{eq:tangent}.  Consider the function
\(
f=\frac{q(x_{1},x_{2},z)}{z\,l(x_{1},x_{2},z)}.
\)
The divisor of this function is 
\[
\div(f)=D_{\varphi}^{+}+Q_{4}^{+}+2O - (O+O_{\omega}+O_{\omega^{2}}+2O+\Obar).
\]
Thus, we have
\[
Q_{4}^{+} - \Obar \sim (O+O_{\omega}+O_{\omega^{2}}) - D_{\varphi}^{+},
\]
where $\Obar$ is the third point of intersection between $C_{u}$ and the tangent line at~$O$.
Let $P_{\Obar}$ be the section in $F^{(6)}_{E_{1}, E_{2}}(k(u))$ corresponding to $\Obar\in C_{u}(k(u))$.  Then, since $(O+O_{\omega}+O_{\omega^{2}}) - D_{\varphi}^{+}$ is invariant under the action $u\mapsto \omega u$, the section $P_{\varphi}^{+}-P_{\Obar}$ is also invariant under this action. This implies that $P_{\varphi}^{+}-P_{\Obar}$ is a $k(t)$-rational section where $t=u^{3}$.

\begin{theorem}\label{thm:F2}
Suppose $E_{1}$ and $E_{2}$ are isogenous by an isogeny $\varphi$ of degree~$3$ over~$k$, and $E_{1}$ and $E_{2}$ do not have complex multiplication.  Let $P_{\varphi}^{(2)}=P_{\varphi}^{+}-P_{\Obar}$ be the section of $F_{E_{1},E_{2}}^{(2)}$ defined as above.  Then, $P_{\varphi}^{(2)}$, $R_{22}$, $R_{33}$, $R_{23},R_{32}$ are linearly independent and generate the Mordell-Weil group $F_{E_{1},E_{2}}^{(2)}(\bar k (t))$.
\end{theorem}

\begin{proof}
Since $E_{1}\sim E_{2}$, but they do not have complex multiplication, we have $\rank\Hom_{\bar k}(E_{1},E_{2})=1$, and we are in case (ii) of Lemma~\ref{lem:3-isog}.
By straightforward calculations, we have
\[
P_{\varphi}^{(2)}=
\left(T^{2} + 4aT + \frac{4}{3}(a^2-3bb'),
\Bigl(bt+\frac{b'}{t}\Bigl)
\bigl(T^2 + 6aT+4(2a^2+bb')\bigr)\right),
\]
where $b'=4a+27b$ and $T=bt-b'/t$. It is easy to show that height of $P_{\varphi}^{(2)}$ equals~$4$.  Since the coordinates of $R_{ij}$ involve the roots of the cubic equation $x_{1}^{3}+a(x_{1}-b)^2=0$, we do not write down the explicit coordinates here, but calculations are straightforward. (We will show them in the numerical examples.) The height matrix with respect to $P_{\varphi}^{(2)}$, $R_{22}$, $R_{33}$, $R_{23},R_{32}$ is given by
\[
\frac{1}{3}\left(\begin{array}{*5r}
12&\phantom{-}0&\phantom{-}0&-3&-3\\ 
0&4&2&0&0\\
0&2&4&0&0\\
-3&0&0&4&2\\
-3&0&0&2&4
\end{array}\right).
\]
Its determinant is $2^4/3$, which equals $2^{4}/3^{2}$ times $\det\Hom_{\bar k}(E_{1},E_{2})=3$.  Thus, $P_{\varphi}^{(2)}$, $R_{22}$, $R_{33}$, $R_{23},R_{32}$ are the generators of $F_{E_{1},E_{2}}^{(2)}(\bar k (t))$.
\end{proof}

\begin{remark}
The height paring of given two rational points is easily computed by the height formula in \cite{Kuwata:can-height}.
\end{remark}

\section{Singular $K3$ surfaces}\label{singular}

A complex $K3$ surface whose Picard number equals the maximum possible number $20$ is called a singular $K3$ surface. Shioda and Inose showed that a complex singular $K3$ surface is isomorphic to an Inose surface $\Ino(E_{1}, E_{2})$ for some elliptic curves $E_{1}$ and $E_{2}$ that have complex multiplication and are isogenous to each other.

\begin{theorem}[\cite{Shioda-Inose}]
There is a one-to-one correspondence between the set of isomorphism classes of complex singular $K3$ surfaces and the set of equivalence classes of positive-definite even integral lattices of rank $2$ with respect to $SL_2(\mathbf{Z})$:
  \[
    \begin{aligned}
      & \Set{ \text{singular $K3$ surfaces over $\Qbar$} } / \text{~isom.}\\
      & \overset{1:1}{\leftrightarrow} \Set{ \left(
        \begin{array}{cc}
          2a & b\\
          b & 2c
        \end{array}
      \right) \ | \ a, b, c \in \mathbf{Z}, \, a, c >0, \ b^2-4ac <0 } 
      \Big/ SL_2(\mathbf{Z}),\\
    \end{aligned}
  \]
which associates a singular $K3$ surface $X$ with its transcendental lattice $T_X$.
\end{theorem}

In fact, a singular $K3$ surface corresponding to the lattice $Q=\left(
  \begin{array}{cc}
    2a & b\\
    b & 2c
  \end{array}
\right)$ is constructed as follows. Let $E_{1}$ and $E_{2}$ be the complex elliptic curves $\mathbf{C}/ \mathbf{Z} \oplus \tau_1 \mathbf{Z}$ and $\mathbf{C}/ \mathbf{Z} \oplus \tau_2 \mathbf{Z}$, respectively, where
\[
  \tau_1 = \frac{-b+\sqrt{b^2-4ac}}{2a}, \quad 
  \tau_2 = \frac{b+\sqrt{b^2-4ac}}{2}.
\]
Then, the Inose surface $\Ino(E_{1}, E_{2})$ is a singular $K3$ surface corresponding to $Q$, which is often denoted by $X_{[a,b,c]}$.

In the following, we study in detail the Mordell-Weil group of $F^{(1)}$ and $F^{(2)}$ for $X_{[3,3,3]}$, $X_{[3,2,3]}$, and $X_{[3,0,3]}$

\subsection{The singular $K3$ surface $X_{[3,3,3]}$}\label{X333}

The transcendental lattice of the singular $K3$ surface $X_{[3,3,3]}$ is given by the matrix
\[
  \left(
    \begin{array}{cc}
      6 & 3\\
      3 & 6
    \end{array}
  \right).
\]
Then, elliptic curves $E_{1}$ and $E_{2}$ are  given by 
\[
   E_{1}=\mathbf{C}/\mathbf{Z} \oplus \mathbf{Z}\omega ,
   \quad
   E_{2}=\mathbf{C}/\mathbf{Z} \oplus \mathbf{Z}(-3\omega^{2}), 
\]
where $\omega = \frac{-1+\sqrt{-3}}{2}$. Their $j$-invariants are given by
\[
  j(E_{1})=0, \quad j(E_{2})=-1288000,
\]
and we take the following Weierstrass forms of them:
\[\setlength{\arraycolsep}{2pt}
\begin{array}{lrcl}
    E_{1} : &y_1^2 & = &x_1^3+6(x_{1}+1)^{2}, \\
    E_{2} : &y_2^2 & = &x_2^3-2(3x_{2}+1)^2.
\end{array}
\]
This is nothing but the case $a=6$, $b=-1$ in \eqref{eq:E1andE2}. Then, $F_{E_{1},E_{2}}^{(1)}$ is given by
\[
F_{E_{1},E_{2}}^{(1)}: Y^2=X^3-27\Bigl(s-506+\frac{9}{s}\Bigr),
\]
after rescaling $X$ and $Y$.  In this case, $F_{E_{1},E_{2}}^{(1)}$ itself has complex multiplication by~$\omega$.

There are two $3$-isogenies from $E_{1}$ to $E_{2}$. One is the $3$-isogeny $\varphi$ in \S\ref{3iso}. We denote it by $\varphi_1$. The other, denoted by $\varphi_2$, is a composite of $\varphi_1$ and multiplication by $\omega$ on $E_{2}$. They are given by
\[
  \varphi_1(x_1, y_1) = (\varphi_x(x_1), \varphi_y(x_1) y_1), \quad
  \varphi_2(x_1,y_1)=  (\omega \varphi_x(x_1), \varphi_y(x_1) y_1),
\]
where
\[
    \varphi_x(x_1)=\frac{x_1^3+8x_1^2+24x_1+24}{x_1^2}, \quad
    \varphi_y(x_1)=-\frac{x_1^3-24x_1-48}{x_1^3} .
\]
These isogenies form a basis of the lattice $\Hom_{\Qbar}(E_{1},E_{2})$ whose Gram matrix is given by \(\dfrac{1}{2}\left(\begin{array}{rr} 6 & 3\\ 3 & 6 \end{array} \right)\). 
The isogeny $\varphi_1$ yields a $\mathbf{Q}$-rational section $P^{(1)}_{\varphi}=P_{\varphi}^{+}-P_{\varphi}^{-}$ of $F^{(1)}_{E_{1},E_{2}}$ (see \S5).  If we let $S=\frac{1}{2}\bigl(\frac{s}{3}+\frac{3}{s}\bigr)$, its coordinates are given by
\begin{multline*}
P^{(1)}_{\varphi} = 
\left(\frac{S^3-93S^2+963S+4129}{64(S-1)},\right.
\\
\left.
\frac{3s(s+3)(S^4-140S^3+4758S^2-13100S+258481)}{256(s-3)^3}\right).
\end{multline*}
The height matrix with respect to sections $P_{\varphi}^{(1)}$ and $[-\omega]P_{\varphi}^{(1)}=\bigl(\omega X,-Y\bigr)$ is given by \(\left( \begin{array}{rr} 6 & 3\\ 3 & 6 \end{array} \right)\).  Thus, by Theorem~\ref{thm:MWrank}, these sections form a basis of the Mordell-Weil group $F_{E_{1},E_{2}}^{(1)}(\Qbar(s))=F_{E_{1},E_{2}}^{(1)}(\mathbf{Q}(\omega)(s))$. 

Next, we consider $F_{E_{1},E_{2}}^{(2)}(\Qbar(t))$.  The explicit formula for the sections $P_{\varphi}^{(2)}$, $R_{22}$, $R_{33}$, $R_{23},R_{32}$ described in Theorem~\ref{thm:F2} are as follows:
\begin{gather*}
P_{\varphi}^{(2)}
=\left(\frac{1}{4}T_{-}^2-6T_{-}+15, 
\frac{1}{8}T_{+}
\bigl(T_{-}^2-36T_{-}+300\bigr)\right)
\\
R_{22}=\Bigl(-15\cdot2^\frac{2}{3}\omega, 
-3\sqrt{-3}\,T_{-}\Bigr), 
\quad R_{33}=[-\omega]R_{22},
\\
R_{23}=\Bigl(-24\omega, -3\sqrt{-3}\,T_{+}\Bigr),
\quad R_{32}=[-\omega]R_{23},
\end{gather*}
where $T_{-}=t-3/t$, and $T_{+}=t+3/t$.

The height matrix with respect to $P_{\varphi}^{(2)}$, $[-\omega]P_{\varphi}^{(2)}$,$R_{22}$, $R_{33}$, $R_{23},R_{32}$ is given by
\[ 
\frac{1}{3}
\left(\begin{array}{*6r} 
12&6&\phantom{-}0&\phantom{-}0&-3&-3\\
6&12&0&0&0&-3\\
0&0&4&2&0&0\\
0&0&2&4&0&0\\
-3&0&0&0&4&2\\
-3&-3&0&0&2&4
\end{array} \right).
\] 
Its determinant is $2^2\cdot3$.  Since $2^{4}/3^{2}\cdot\det\Hom_{k}(E_{1},E_{2})=2^{4}/3^{2}\cdot\det\left(\frac{1}{2}\begin{pmatrix} 6 & 3 \\ 3 & 6 \end{pmatrix}\right)=2^2\cdot3$, the above sections generate $F_{E_{1},E_{2}}^{(2)}(\Qbar(t))$ by Theorem~\ref{thm:F2-Shioda}.

\subsection{The singular $K3$ surface $X_{[3,2,3]}$}\label{X323}
The transcendental lattice of the singular $K3$ surface $X_{[3,2,3]}$ is given by the matrix
\[
  \left(
    \begin{array}{cc}
      6 & 2\\
      2 & 6
    \end{array}
  \right).
\]
Then, elliptic curves $E_{1}$ and $E_{2}$ are  given by $\mathbf{C}/ \mathbf{Z} \oplus \tau_1 \mathbf{Z}$ and $\mathbf{C}/ \mathbf{Z} \oplus \tau_2 \mathbf{Z}$, where
\[
  \tau_1 = \frac{-1+2\sqrt{-2}}{3}, \quad 
  \tau_2 = 1+2\sqrt{-2}.
\]
The $j$-invariants of $E_{1}$ and $E_{2}$ are given by
\[
j(E_{1})= 26125000-18473000\sqrt{2},\quad
j(E_{2})= 26125000+18473000\sqrt{2}.
\]
We thus work on the base field $k=\mathbf{Q}(\sqrt{2})$.  We choose $E_{1}$ and $E_{2}$ such that their Weierstrass forms are given by
\begin{align*}
E_{1}: y_{1}^{2} = x_{1}^3+6(3-\sqrt{2})x_{1}^2+9(3+2\sqrt{2}+3)x_{1}, \\
E_{2}: y_{2}^{2} = x_{2}^3+6(3+\sqrt{2})x_{2}^2+9(3-2\sqrt{2}+3)x_{2}.
\end{align*}
They are so-called $\mathbf{Q}$-curves; they are Galois conjugates and isogenous to each other. 

Both $E_{1}$ and $E_{2}$ have complex multiplication by $\mathbf{Z}[1+2\sqrt{-2}]$.  Let $K=\mathbf{Q}(\sqrt{-2})$, and let $H=kK=\mathbf{Q}(\sqrt{2},\sqrt{-2})$ be the Hilbert class field of $K$. 
They admit two isogenies $\varphi_{1}, \varphi_{2}:E_{1}\to E_{2}$ of degree~$3$ defined over $H$ given by
\[
  \varphi_1(x_1, y_1) 
  = (\varphi_{1,x}(x_1), \varphi_{1,y}(x_1) y_1), \quad
  \varphi_2(x_1,y_1)
  = (\varphi_{2,x}(x_1), \varphi_{2,y}(x_1) y_1),
\]
where
\begin{align*}
&\varphi_{1,x}(x_{1}) = \frac{(1-2\sqrt{-2})x_{1}(x_{1}+3-3\sqrt{-2})^2}
{9(x_{1}+3+3\sqrt{-2}+2\sqrt{2}+4\sqrt{-1})^2},
\\
&\varphi_{2,x}(x_{1}) = \frac{(1+2\sqrt{-2})x_{1}(x_{1}+3+3\sqrt{-2})^2}
{9(x_{1}+3-3\sqrt{-2}+2\sqrt{2}-4\sqrt{-1})^2}.
\end{align*}
The isogeny $\varphi_i$ yields a $H$-rational section $P_{\varphi_{i}}^{+}$ of $F^{(6)}_{E_{1},E_{2}}$ for $i=1,2$.  

Let $\hat \varphi_{i}:E_{2}\to E_{1}$ be the dual isogeny of $\varphi_{i}$ for $i=1,2$.  The endomorphisms $\hat\varphi_{2}\circ\varphi_{1}$ and $\hat\varphi_{1}\circ\varphi_{2}$ correspond to the complex multiplication of $E_{1}$ by $1\pm 2\sqrt{-2}$.

The surface $F_{E_{1},E_{2}}^{(6)}$ is defined over $k$ and given by
\[
F_{E_{1},E_{2}}^{(6)}:
Y^2 = X^3 + \frac{575}{12}X
+\biggl(\frac{u^{6}}{(1-\sqrt{2})^3}-\frac{34937}{108}
-\frac{(1-\sqrt{2})^3}{u^6}
\biggr).
\]
By letting $s'=s/(1-\sqrt{2})^{3}=u^{1/6}/(1-\sqrt{2})^{3}$, the Weierstrass equation of the surface $F_{E_{1},E_{2}}^{(1)}$ is given by 
\[
F_{E_{1},E_{2}}^{(1)}:
Y^2 = X^3 + \frac{575}{12}X
+\Bigl(s'-\frac{34937}{108}
-\frac{1}{s'}
\Bigr).
\]
With this model, the $X$-coordinate of the section $P^{(1)}_{\varphi_{1}}=P_{\varphi_{1}}^{+}-P_{\varphi_{1}}^{-}$ of $F_{E_{1},E_{2}}^{(1)}$ (see \S5) is given by
\[
X(P_{\varphi_{1}}^{(1)})=
\frac{-\sqrt{-1}}
{12(2-\sqrt{-1})^8(S+2\sqrt{-1})}
(3S'^3 +c_{2}S'^2 + c_{1}S' + c_{0}),
\]
where
\begin{gather*}
S'=s'-\frac{1}{s'}, \quad
c_{2}=-42(23-10\sqrt{-1}),\quad
c_{1}= 2(9402-13685\sqrt{-1}),\\
c_{0}= - 4(61663+50160\sqrt{-1}).
\end{gather*}
The $Y$-coordinate can be obtained easily, but it is rather complicated and we do not include here. The section $P_{\varphi_{2}}^{(1)}$ is the image of $P_{\varphi_{1}}^{(1)}$ under the complex conjugate $\sqrt{-1}\mapsto -\sqrt{-1}$.  As in the previous example, the height matrix with respect to the sections $P_{\varphi_{1}}^{(1)}$ and $P_{\varphi_{2}}^{(1)}$ is given by 
\(\left( \begin{array}{rr} 6 & 2\\ 2 & 6 \end{array} \right)\), and they form a basis of the Mordell-Weil lattice $F_{E_{1},E_{2}}^{(1)}(\bar k(s))=F_{E_{1},E_{2}}^{(1)}(H(s))$.

Let $L=H\bigl(\sqrt{1-\sqrt{2}}\bigr)=\mathbf{Q}\bigl(\sqrt{1-\sqrt{2}},\sqrt{-1}\bigr)$.  Then, all the $2$-torsion points of $E_{1}$ and $E_{2}$ are defined over $L$,  and thus, the field of definition of $F_{E_{1},E_{2}}^{(2)}(\bar k(t))$ is~$L$.  If we let $t'=u^{1/3}/(1-\sqrt{2})^{3/2}$, the Weierstrass equation of $F_{E_{1},E_{2}}^{(2)}$ is given by
\[
F_{E_{1},E_{2}}^{(2)}:
Y^2 = X^3 + \frac{575}{12}X
+\Bigl(t'^{2}-\frac{34937}{108}
-\frac{1}{t'^{2}}
\Bigr).
\]
Let $P_{\varphi_{i}}^{(2)}=P_{\varphi_{i}}^{+}-P_{\Obar}$ for $i=1,2$ (see \S\ref{sec:F2}).  The $X$-coordinate of the section $P_{\varphi_{1}}^{(2)}$ is given by
\[
X(P_{\varphi_{1}}^{(2)}) = 
\frac{-(1+\sqrt{-1})}{(2(1+2\sqrt{-1})^2)}
\Bigl(T^2 +c_{1} T +c_{0}\Bigr),
\]
where 
\begin{align*}
&T=t+\sqrt{-1}/t,\\
&c_{1}=\sqrt{1-\sqrt{2}}\Bigl(9+13\sqrt{-1}-2\sqrt{2}+11\sqrt{-2}\Bigr),\quad
c_{0}=\frac{1}{6}\bigl(161-97\sqrt{-1}\bigr).  
\end{align*}
Since $P_{\Obar}$ is defined over $k=\mathbf{Q}(\sqrt{2})$, $P_{\varphi_{2}}^{(2)}$ is also the image of  $P_{\varphi_{1}}^{(2)}$ by the complex conjugate $\sqrt{-1}\mapsto -\sqrt{-1}$.

The explicit formula for the sections $R_{22}$ described in Theorem~\ref{thm:F2} is
\[
R_{22}=\Bigl(
(1-2\sqrt{-1})(2+\sqrt{2}-\sqrt{-2})\sqrt{1-\sqrt{2}}
-\frac{(1-2\sqrt{-1})^{4}}{6}, 
T\Bigr).
\]
Let $\sigma\in\Gal(L/\mathbf{Q})$ be the automorphism given by $\sqrt{1-\sqrt{2}} \mapsto -\sqrt{1-\sqrt{2}}$, and $\gamma$ the complex conjugate $\sqrt{-1}\mapsto \sqrt{-1}$.  Then, we have
\[
R_{33}=-\sigma(R_{22}), \quad 
R_{23}=\gamma(R_{22}), \quad
R_{32}=\gamma(R_{33}).
\]

The height matrix with respect to $P_{\varphi_{1}}^{(2)}$, $P_{\varphi_{2}}^{(2)}$, $R_{22}$, $R_{33}$, $R_{23}$ and $R_{32}$ is given by
\[
\frac{1}{3}\left(\begin{array}{*6r}
12&3&0&0&-3&-3\\ 
3&12&-3&-3&0&0\\ 
0&-3&4&2&0&0\\
0&-3&2&4&0&0\\
-3&0&0&0&4&2\\
-3&0&0&0&2&4
\end{array}\right).
\]
Its determinant is $2^7/3^2=2^{4}/3^{2}\cdot\det\Hom_{k}(E_{1},E_{2})$, and we can see that the above sections generate $F_{E_{1},E_{2}}^{(2)}(\bar k(t))$ as in the previous example.

\subsection{The singular $K3$ surface $X_{[3,0,3]}$}\label{X303}

The transcendental lattice of the singular $K3$ surface
$X_{[3,0,3]}$ is given by the
matrix
\[
  \left(
    \begin{array}{cc}
      6 & 0\\
      0 & 6
    \end{array}
  \right).
\]
In this case, the complex elliptic curves $E_{1}/\mathbf{C}$ and
$E_{2}/\mathbf{C}$ are given by
\[
  E_{1} : \mathbf{C}/\mathbf{Z} \oplus \mathbf{Z}(3\sqrt{-1}), 
  \quad 
  E_{2} : \mathbf{C}/\mathbf{Z} \oplus \mathbf{Z}\sqrt{-1}.
\]
Their $j$-invariants are given by
\[
	j(E_{1})=76771008+44330496\sqrt{3}, \quad
 	j(E_{2})=1728.
\]
We work with the base field $k=\mathbf{Q}(\sqrt{3})$, and we take the following Weierstrass forms, which corresponds to the formula \eqref{eq:E1andE2} with $a=9(2+\sqrt{3})$, $b=(1-\sqrt{3})/3$. (The equation of $E_{2}$ is scaled differently from \eqref{eq:E1andE2}.)
\[\setlength{\arraycolsep}{2pt}
  \begin{array}{lrcl}
    E_{1} : & y_1^2 &= &x_{1}^3+(2+\sqrt{3})(3x_{1}-1+\sqrt{3})^2\\
    E_{2} : & y_2^2 &= &x_{2}^3-27(2+\sqrt{3})(x_{2}-9-3\sqrt{3})^2  
    \end{array}
\]
The curves $E_{1}$ and $E_{2}$ are isogenous with complex multiplication by $\mathbf{Q}(\sqrt{-1})$, and there are two $3$-isogenies from $E_{1}$ to $E_{2}$. One is obtained by a coordinate changes of the $3$-isogeny $\varphi$ in \S\ref{3iso}. We denote it by $\varphi_1$. The other, denoted by $\varphi_2$, is a composite of $\varphi_1$ and the complex multiplication $[\sqrt{-1}]$ on $E_{2}$.  With our equation of $E_{2}$, the multiplication-by-$\sqrt{-1}$ map on $E_{2}$ is given by $(x_{2},y_{2}) \mapsto (-x_{2}+36+18\sqrt{3},\sqrt{-1}y_{2})$.

The $3$-isogenies $\varphi_1, \varphi_2$ are given by
\begin{align*}
&\varphi_{1}(x_1, y_1) = \left( \varphi_{x}(x_1), \varphi_{y} (x_1)y_1 \right), \\
&\varphi_{2}(x_1, y_1) = \left( -\varphi_{x}(x_1)+36+18\sqrt{3}, 
\sqrt{-1} \varphi_{y}(x_1) y_1 \right),
\end{align*}
where
\[\label{Phi1xy}
  \begin{aligned}
  &\varphi_{x}(x_1) = 
  \frac{x_{1}^3+12(2+\sqrt{3})x_{1}^2+12(1+\sqrt{3})x_{1}+8}
  							{x_{1}^2},\\
  &\varphi_{y}(x_1) = -\frac{x_{1}^3-12(1+\sqrt{3})x_{1}-16
    {x_{1}^3}}.
  \end{aligned}
\]
Over the base field $k$, the equation of $F_{E_{1}, E_{2}}^{(6)}$ is given by
\[\label{eq:F6-6006}
F_{E_{1}, E_{2}}^{(6)}: Y^2 = 
X^3 - (387+224\sqrt{3})X 
+ 3(3+2\sqrt{3})u^{6}+\frac{45+26\sqrt{3}}{9u^{6}},
\]
after replacing $u$ by $3u$.  By letting $s'=s(3+2\sqrt{3})/9=u^{6}(3+2\sqrt{3})/9$, the equation of $F_{E_{1},E_{2}}^{(1)}$ is given by 
\[
F_{E_{1},E_{2}}^{(1)}:
Y^2 = X^3 -(387+224\sqrt{3})X
+(7+4\sqrt{3})\Bigl(s'+\frac{1}{s'}\Bigr).
\]
Using the method described in \S\ref{sec:F1}, the section $P_{\varphi_{1}}^{(1)}$ of $F_{E_{1},E_{2}}^{(1)}$ is given by
\[
P_{\varphi_{1}}^{(1)}=
\biggl(\frac{c_{3}S'^{3}+c_{2}S'^{2}+c_{1}S' +c_{0}}
    {144(S'+2)},
\frac{s'(s'-1)\bigl(d_{4}S'^{4} + d_{3}S'^{3} +d_{2}S'^{2} +d_{1}S' +d_{0}\bigr)}{1728\left(s'+1\right)^3}
    \biggr),
\]
where 
\begin{align*}
&S'=s'+\frac{1}{s'},\\
&c_{3}=(2-\sqrt{3})^{2},\quad
c_{2}=- 42,\quad
c_{1}=12(91+36\sqrt{3}),\quad
c_{0}=- 8(1267+680\sqrt{3}),\\
&d_{4}=(2-\sqrt{3})^{3},\quad
d_{3}=-4(25-12\sqrt{3}),\quad
d_{2}=24(107+15\sqrt{3}),
\\
&d_{1}=16(461+444\sqrt{3}),\quad
d_{0}=-16(54676+32091\sqrt{3}).
\end{align*}
If we write $P_{\varphi_{1}}^{(1)}=(x(s'), y(s'))$, then we have $P_{\varphi_{2}}^{(1)}=(-x(-s'), \sqrt{-1}\, y(-s'))$. The height matrix with respect to the sections $P_{\varphi_{1}}^{(1)}$ and $P_{\varphi_{2}}^{(1)}$ is given by
\(\left(
    \begin{array}{cc}
      6 & 0\\
      0 & 6
    \end{array}
  \right)\), and they form a basis of $F^{(1)}_{E_{1}, E_{2}}(\bar k(s))=F^{(1)}_{E_{1}, E_{2}}(\mathbf{Q}(\sqrt{3},\sqrt{-1})(s))$.  

Next, we consider $F_{E_{1},E_{2}}^{(2)}(\bar k(t))$.  The field of definition of $2$-torsion subgroups $E_{1}[2]$ and $E_{2}[2]$ is $\mathbf{Q}(\sqrt{3+2\sqrt{3}})$.  In order to compute generators of $F_{E_{1},E_{2}}^{(2)}(\bar k(t))$, we need to work with the field $L=\mathbf{Q}\bigl(\sqrt{3+2\sqrt{3}},\sqrt{-1}\bigr)$.

The explicit formulas for the sections $P_{\varphi_{1}}^{(2)}$, $P_{\varphi_{2}}^{(2)}$, $R_{22}$, $R_{33}$, $R_{23},R_{32}$ described in Theorem~\ref{thm:F2} are as follows:
\begin{align*}
P_{\varphi_{1}}^{(2)}
&=\biggl(\frac{1}{6}
\Bigl(\sqrt{3}\,T_{+}^2+6(1+\sqrt{3})\,T_{+}+42+20\sqrt{3}\Bigr), 
\\ &\hspace{2.4cm}
\frac{\alpha}{12} T_{-}
\Bigl((1-\sqrt{3})\,T_{+}^2-6\sqrt{3}\alpha\, T_{+}
+112+56\sqrt{3}\Bigr)\biggr)
\\
P_{\varphi_{2}}^{(2)}
&=\biggl(\frac{1}{6}
\Bigl(\sqrt{3}\,T_{-}^2+6\sqrt{-1}(1+\sqrt{3})\,T_{-}
-42-20\sqrt{3}\Bigr), 
\\ &\hspace{2cm}
-\frac{\alpha}{12} T_{+}
\Bigl(-(1-\sqrt{3})\,T_{+}^2-6\sqrt{-3}\alpha\, T_{+}
-112-56\sqrt{3}\Bigr)\biggr)
\\
R_{22}&=\left(-7-4\sqrt{3}+2(1+\sqrt{3})\alpha, 
(2+\sqrt{-3})T_{+}\right), 
\\
R_{33}&=\left(-7-4\sqrt{3}-2(1+\sqrt{3})\alpha, 
-(2+\sqrt{-3})T_{+}\right),
\\
R_{23}&=\left(7+4\sqrt{3}-2(1+\sqrt{3})\alpha, 
(2+\sqrt{-3})T_{-}\right), 
\\
R_{32}&=\left(7+4\sqrt{3}+2(1+\sqrt{3})\alpha, 
-(2+\sqrt{-3})T_{-}\right),
\end{align*}
where
\[
\alpha = \sqrt{3+2\sqrt{3}}, \quad
T_{+} = \frac{3t}{\alpha} + \frac{\alpha}{3t}, \quad
T_{-} = \frac{3t}{\alpha} - \frac{\alpha}{3t}.
\]
The height matrix with respect to $P_{\varphi_{1}}^{(2)}$, $P_{\varphi_{2}}^{(2)}$,$R_{22}$, $R_{33}$, $R_{23},R_{32}$ is given by
\[ 
\frac{1}{3}
\left(\begin{array}{*6r} 
12&0&\phantom{-}0&\phantom{-}0&-3&-3\\
0&12&-3&-3&0&0\\
0&-3&4&2&0&0\\
0&-3&2&4&0&0\\
-3&0&0&0&4&2\\
-3&0&0&0&2&4
\end{array} \right).
\] 
Its determinant is $2^4=2^{4}/3^{2}\cdot\det\Hom_{k}(E_{1},E_{2})$ as expected, and we can see that the above sections generate $F_{E_{1},E_{2}}^{(2)}(\bar k(t))$ as before.

\def\arXiv#1{arXiv:\href{http://arXiv.org/abs/#1}{#1}}
\providecommand{\bysame}{\leavevmode\hbox to3em{\hrulefill}\thinspace}
\providecommand{\MR}{\relax\ifhmode\unskip\space\fi MR }
\providecommand{\MRhref}[2]{%
  \href{http://www.ams.org/mathscinet-getitem?mr=#1}{#2}
}
\providecommand{\href}[2]{#2}

\end{document}